\documentclass[10pt,leqno,letterpaper]{article}
\usepackage[utf8]{inputenc}
\usepackage{amsmath}
\usepackage{amsfonts}
\usepackage{amssymb}
\usepackage{graphicx}
\usepackage{verbatim}
\usepackage{mathrsfs}
\usepackage{upref,amsthm,amsxtra,exscale}
\usepackage{cite}
\usepackage[colorlinks=true,urlcolor=blue,
citecolor=red,linkcolor=blue,linktocpage,pdfpagelabels,
bookmarksnumbered,bookmarksopen]{hyperref}
\usepackage{upgreek}

\newtheorem{theorem}{Theorem}[section]
\newtheorem{corollary}[theorem]{Corollary}
\newtheorem{remark}[theorem]{Remark}

\newtheorem{lemma}[theorem]{Lemma}
\newtheorem{proposition}[theorem]{Proposition}

\newtheorem{example}[theorem]{Example}

\numberwithin{equation}{section}

\def\r{\mathbb{R}}
\def\rn{\mathbb{R}^N}
\def\z{\mathbb{Z}}

\def\n{\mathbb{N}}
\def\cc{\mathbb{C}}
\def\eps{\varepsilon}
\def\rh{\rightharpoonup}

\def\irn{\int_{\r^N}}
\def\vp{\varphi}
\def\o{\Omega}
\def\t{\Theta}
\def\bf{\boldsymbol}

\def\cC{\mathcal{C}}

\def\cH{\mathcal{H}}

\def\cJ{\mathcal{J}}

\def\cM{\mathcal{M}}
\def\cN{\mathcal{N}}

\def\cP{\mathcal{P}}

\def\cW{\mathcal{W}}

\def\supp{\text{supp}}

\def\what{\widehat}
\def\d{\,\mathrm{d}}
\def\e{\mathrm{e}}
\def\i{\mathrm{i}}

\author{Mónica Clapp\footnote{M. Clapp was supported by CONACYT (Mexico) through the research grant A1-S-10457.} \ and Angela Pistoia\footnote{A. Pistoia was partially supported by INDAM-GNAMPA funds and Fondi di Ateneo ``Sapienza" Universit\`a di Roma (Italy).}}
\title{Pinwheel solutions to Schrödinger systems}
\date{\today}

\begin{document}
\maketitle

\begin{abstract}
We establish the existence of positive segregated solutions for competitive nonlinear Schrödinger systems in the presence of an external trapping potential, which have the property that each component is obtained from the previous one by a rotation, and we study their behavior as the forces of interaction become very small or very large. 

As a consequence, we obtain optimal partitions for the Schrödinger equation by sets that are linearly isometric to each other. 
\end{abstract}


\section{Introduction}
Consider the nonlinear Schrödinger system
\begin{equation} \label{eq:system0}
\begin{cases}
-\Delta u_i+ V_i(x)u_i = |u_i|^{2p-2}u_i+\sum\limits_{\substack{j=1\\j\neq i}}^\ell\beta_{ij}|u_j|^p|u_i|^{p-2}u_i, \\
u_i\in H^1(\rn),\quad u_i>0,\qquad i=1,\ldots,\ell,
\end{cases}
\end{equation}
where $N\geq 2$, $p>1$ and $p<\frac{N}{N-2}$ if $N\geq3$, $\beta_{ij}=\beta_{ji}\in\r$, and $V_i\in\cC^0(\rn).$

For the cubic nonlinearity ($p=2$) in dimensions $N=2,3$ this system arises in the study of Bose-Einstein condensation for a mixture of $\ell$ different states which overlap in space. It has been widely studied in the last two decades. Most work has been done in the autonomous case (i.e., for constant $V_i$). We refer the reader to the   recent paper   \cite{liweiwu} where the authors provide an exhaustive list of  references. The non-autonomous case turns out to be much more difficult. Some results have been recently obtained by Peng and Wang \cite{pw}, Pistoia and Vaira \cite{pv}, and Li, Wei and Wu \cite{liweiwu}.

The system \eqref{eq:system0} for a more general subcritical nonlinearity in higher dimensions has been much less studied. Even if it does not have an immediate physical motivation, finding a solution in this general setting is a quite interesting and challenging problem from a mathematical point of view. To our knowledge, the only result so far is that by Gao and Guo \cite{gg} who proved the existence of infinitely many solutions for the system of only two equations ($\ell=2$) when the coupling parameter $\beta_{12}$ is negative, and both equations have a common potential $V_1=V_2$, which does not enjoy any symmetry properties, but satisfies suitable decay assumptions at infinity. However, nothing is said about the sign of the solutions.

Here we study \eqref{eq:system0} in a fully symmetric setting, namely we consider the nonlinear Schrödinger system 
\begin{equation} \label{eq:system}
\begin{cases}
-\Delta u_i+ V(x)u_i = |u_i|^{2p-2}u_i+\beta\sum\limits_{\substack{j=1\\j\neq i}}^\ell|u_j|^p|u_i|^{p-2}u_i, \\
u_i\in H^1(\rn),\quad u_i>0,\qquad i=1,\ldots,\ell,
\end{cases}
\end{equation}
where $N\geq 4$, $1<p<\frac{N}{N-2}$,   $\beta<0$ and $V\in\cC^0(\rn)$  satisfies the following assumptions for some $n\in\n$:
\begin{itemize}
\item[$(V_1)$] $V$ is radial.
\item[$(V_2)$] $0<\inf_{x\in\rn}V(x)$ and $V(x)\to V_\infty>0$ as $|x|\to\infty$.
\item[$(V_3^n)$] There exist $C_0,R_0>0$ and $\lambda\in(0,2\sin\frac{\pi}{\ell n})$ such that
$$V(x)\leq V_\infty-C_0\e^{-\lambda\sqrt{V_\infty}|x|}\qquad\text{for every \ }x\in\rn\text{ \ with \ }|x|\geq R_0.$$
\end{itemize}
We look for  \emph{fully nontrivial} solutions to \eqref{eq:system}, i.e., solutions with all components $u_i$ different from zero.
Set
$$\|u\|^2_V:=\irn(|\nabla u|^2+V(x)u^2).$$
We prove the following results.

\begin{theorem} \label{thm:existence}
Let $n\in\n$ and assume that $V$ satisfies $(V_1)$, $(V_2)$ and $(V_3^n)$. Then the system \eqref{eq:system} has a fully nontrivial solution $\bf u=(u_1,\ldots,u_\ell)$ satisfying
\begin{equation} \label{eq:symmetries}
\begin{cases}
u_1(\mathrm{e}^{2\pi\mathrm{i}/n}z,\theta y)=u_1(z,y)&\text{for every \ } \theta\in O(N-2), \\
u_{j+1}(z,y)=u_1(\mathrm{e}^{2\pi\mathrm{i}j/\ell n}z,y) &\text{for every \ } j=1,\ldots,\ell-1,
\end{cases}
\end{equation}
and every $(z,y)\in\cc\times\r^{N-2}\equiv\rn$. This solution has least energy among all nontrivial solutions satisfying \eqref{eq:symmetries}. Furthermore, the energy of each component satisfies
$$\frac{p-1}{2p}\|u_i\|_V^2<n\,\mathfrak{c}_\infty,$$
where $\mathfrak{c}_\infty$ is the ground state energy of the Schrödinger equation
\begin{equation} \label{eq:limit_problem}
-\Delta u + V_\infty u = |u|^{2p-2}u,\qquad u\in H^1(\rn).
\end{equation}
\end{theorem}

As usual, $O(N-2)$ denotes the group of linear isometries of $\r^{N-2}$. The symmetries \eqref{eq:symmetries} of the solutions given by Theorem \ref{thm:existence} suggests calling them \emph{pinwheel solutions}.

Since the potential $V$ is assumed to be radial, using the compactness of the embedding of the subspace of radial functions in $H^1(\rn)$ into $L^{2p}(\rn)$ and following the argument given in \cite[Theorem 1.1]{csz}, it is easy to see that the system \eqref{eq:system} has a solution all of whose components are radial. Note however that, if $\bf u=(u_1,\ldots,u_\ell)$ satisfies \eqref{eq:system} and \eqref{eq:symmetries} and some component $u_i$ is radial, then $u_1=\cdots=u_\ell=:u$ and $u$ is a nontrivial solution of the equation
$$-\Delta u + V(x)u = (1+(\ell-1)\beta)|u|^{2p-2}u,\qquad u\in H^1(\rn).$$
Therefore, if $1+(\ell-1)\beta\leq 0$, a nontrivial solution to the system \eqref{eq:system} satisfying \eqref{eq:symmetries} cannot be radial. In fact, more can be said. The following result, combined with Theorem \eqref{thm:existence}, yields multiple positive nonradial solutions when the assumption $(V_3^n)$ is satisfied for large enough $n$.

\begin{proposition}\label{prop:symmetry_breaking}
Let $\beta\leq-\frac{1}{\ell-1}$ and, for some $m,q\in\n$, let $\bf u_m,\bf u_q$ be solutions to \eqref{eq:system} satisfying \eqref{eq:symmetries} with $n=\ell^m$ and $n=\ell^q$ respectively. If $m\neq q$, then $\bf u_m\neq\bf u_q$.
\end{proposition}

One may wonder if the solution given by Theorem \ref{thm:existence} for $\beta\in(-\frac{1}{\ell-1},0)$ is radial or not. The following result gives a partial answer in terms of the nonautonomous Schrödinger equation \eqref{eq:equation}. Namely, if the least energy solutions to this equation that satisfy \eqref{eq:invariance} are nonradial, then the solutions to the system \eqref{eq:system} satisfying \eqref{eq:symmetries} are nonradial for $\beta$ close enough to $0$.

\begin{theorem} \label{thm:beta_to_0}
Let $n\in\n$ and assume that $V$ satisfies $(V_1)$, $(V_2)$ and $(V_3^n)$. Let $\bf u_k=(u_{k,1},\ldots,u_{k,\ell})$ be a least energy fully nontrivial solution to \eqref{eq:system} and \eqref{eq:symmetries} with $\beta=\beta_k$. Assume that $\beta_k<0$ and $\beta_k\to 0$ as $k\to\infty$. Then, after passing to a subsequence, $u_{k,j}\to u_{0,j}$ strongly in $H^1(\rn)$,\, $u_{0,j}\geq 0$,\, $\bf u_0=(u_{0,1},\ldots,u_{0,\ell})$ satisfies \eqref{eq:symmetries}, $u_{0,j}$ is a nontrivial solution to the equation
\begin{equation} \label{eq:equation}
-\Delta u + V(x)u = |u|^{2p-2}u,\qquad u\in H^1(\rn),
\end{equation}
and $u_{0,j}$ has least energy among all solutions to \eqref{eq:equation} satisfying
\begin{equation} \label{eq:invariance}
u(\mathrm{e}^{2\pi\mathrm{i}/n}z,\theta y)=u(z,y)\quad\text{for every \ } \theta\in O(N-2), \ (z,y)\in\rn.
\end{equation}
Furthermore,
$$\frac{p-1}{2p}\|u_{0,j}\|^2_V<n\,\mathfrak{c}_\infty.$$
\end{theorem}

Next, we describe the behavior of the solutions given by Theorem \ref{thm:existence} as $\beta\to -\infty$. As shown by Conti, Terracini and Verzini \cite{ctv1, ctv2} and Chang, Lin, Lin and Lin \cite{clll}, there is a connection between variational elliptic systems with strong competitive interaction and optimal partition problems.

We shall call an $\ell$-tuple $(\o_1,\ldots,\o_\ell)$ of nonempty open subsets of $\rn$ an \emph{$(n,\ell)$-pinwheel partition} of $\rn$ if $\o_i\cap\o_j=\emptyset$ whenever $i\neq j$ and it satisfies following two symmetry conditions: 
\begin{itemize}
\item[$(S_1)$] $\o_{j+1}=\{(z,y)\in\cc\times\r^{N-2}:(\mathrm{e}^{2\pi\mathrm{i}j/\ell n}z,y)\in\o_1\}$ for each $j=1,\ldots,\ell-1$.
\item[$(S_2)$] If $(z,y)\in\o_1$ then $(\mathrm{e}^{2\pi\mathrm{i}/n}z,\theta y)\in\o_1$ for every $\theta\in O(N-2)$.
\end{itemize}
We denote the set of all $(n,\ell)$-pinwheel partitions by $\cP^{n}_\ell$. \ If $\o$ is an open subset of $\rn$ satisfying $(S_2)$, a minimizer for
$$\inf_{u\in\cM_{\o}}\frac{p-1}{2p}\|u\|_V^2=:\mathfrak{c}_{\o}$$
on the Nehari manifold
\begin{align} \label{eq:bdd_nehari}
\cM_{\o}:=&\{u\in H^1_0(\o):u\neq 0, \ \|u\|_V^2=\irn|u|^{2p},\text{ \ and \ } \\ 
&u(\mathrm{e}^{2\pi\mathrm{i}/n}z,\theta y)=u(z,y)\text{ \ for all \ } \theta\in O(N-2)\text{ \ and \ }(z,y)\in\o\}, \nonumber
\end{align}
is a \emph{least energy solution} to the problem
\begin{equation} \label{eq:equation_bdd}
\begin{cases}
-\Delta u + V(x)u = |u|^{2p-2}u,\qquad u\in H_0^1(\o),\\
u(\mathrm{e}^{2\pi\mathrm{i}/n}z,\theta y)=u(z,y)\quad\text{for every \ } \theta\in O(N-2), \ (z,y)\in\o.
\end{cases}
\end{equation}
We say that $(\o_1,\ldots,\o_\ell)$ is an \emph{optimal $(n,\ell)$-pinwheel partition} for equation \eqref{eq:equation} if $\mathfrak{c}_{\o_j}$ is attained on $\cM_{\o_j}$ and
$$\sum_{j=1}^\ell\mathfrak{c}_{\o_j}=\inf_{(\t_1,\ldots,\t_\ell)\in\cP^{n}_\ell}\sum_{j=1}^\ell\mathfrak{c}_{\t_j}.$$

\begin{theorem} \label{thm:partition}
Let $n\in\n$ and assume that $V$ satisfies $(V_1)$, $(V_2)$ and $(V_3^n)$. Let $\bf u_k=(u_{k,1},\ldots,u_{k,\ell})$ be a least energy fully nontrivial solution to \eqref{eq:system} and \eqref{eq:symmetries} with $\beta=\beta_k$. Assume that $\beta_k\to -\infty$ as $k\to\infty$. Then, after passing to a subsequence, 
\begin{itemize}
\item[$(i)$] $u_{k,j}\to u_{\infty,j}$ strongly in $H^1(\rn)$,\, $u_{\infty,j}\geq 0$,\, $u_{\infty,j}\neq 0$,\, $u_{\infty,i}u_{\infty,j}=0$ if $i\neq j$,\, $\bf u_\infty=(u_{\infty,1},\ldots,u_{\infty,\ell})$ satisfies \eqref{eq:symmetries}, and
\begin{equation*}
\irn \beta_k u_{k,j}^p u_{k,i}^p\to 0 \text{ as } k\to \infty\quad \text{whenever } i\neq j.
\end{equation*}
\item[$(ii)$] $u_{\infty,j}\in\cC^0(\rn)$,\, the restriction of $u_{\infty,j}$ to the open set $\o_j:=\{x\in\rn:u_{\infty,j}(x)>0\}$ is a least energy solution to the problem \eqref{eq:equation_bdd} in $\o_j$, and $(\o_1,\ldots,\o_\ell)$ is an optimal $(n,\ell)$-pinwheel partition for equation \eqref{eq:equation}.
\item[$(iii)$] $\rn\smallsetminus\bigcup_{j=1}^\ell\Omega_j=\mathscr R\cup\mathscr S$, where $\mathscr R\cap\mathscr S=\emptyset$, $\mathscr R$ is an  $(m-1)$-dimensional $\cC^{1,\alpha}$-submanifold of $\rn$ and $\mathscr S$ is a closed subset of $\rn$ with Hausdorff measure $\leq m-2$. Furthermore, if $\xi\in\mathscr R$, there exist $i,j$ such that
$$\lim_{x\to\xi^+}|\nabla u_i(x)|=\lim_{x\to\xi^-}|\nabla u_j(x)|\neq 0,$$
where $x\to\xi^\pm$ are the limits taken from opposite sides of $\mathscr R$ and, if $\xi\in\mathscr S$, then
$$\lim_{x\to\xi}|\nabla u_j(x)|=0\qquad\text{for every \ }j=1,\ldots,\ell.$$
\item[$(iv)$] If $\ell=2$, then \ $u_{\infty,1}-u_{\infty,2}$ \ is a sign-changing solution to equation \eqref{eq:equation} satisfying \eqref{eq:invariance}.
\end{itemize}
\end{theorem}

Note that $(iii)$ implies that the partition exhausts $\rn$, i.e., $\rn=\bigcup_{j=1}^\ell\overline{\o}_j$. Therefore, every $\o_j$ is unbounded.

The regularity properties of optimal partitions have been established, in different settings, for instance in \cite{cl, cpt, nttv, sttz, tt} and some of the references therein. 

Theorem \ref{thm:partition} establishes the existence of optimal partitions having an additional property: each set of the partition is obtained from any other by means of a linear isometry. Pinwheel partitions are an example of this type of partitions, but others are conceivable. In Section \ref{sec:preliminaries} we present a general symmetric variational setting for the system \eqref{eq:system} that produces other examples.

The existence of sign-changing solutions to equation \eqref{eq:equation} having the additional property that their negative part is obtained from the positive one by means of a linear isometry and a change of sign has been established in \cite{cs}. This includes those given by Theorem \ref{thm:partition}$(iv)$. The tool for producing this type of solutions is a homomorphism from some group of linear isometries of $\rn$ onto the group with two elements. As shown is Section \ref{sec:preliminaries} this tool also serves to get positive solutions of the system \eqref{eq:system} for $\ell=2$ with the property that each component is obtained from the other by composition with a linear isometry. The general tool for obtaining a similar result for the system \eqref{eq:system} of $\ell$ equations is a homomorphism into the group of permutations of a set of $\ell$ elements.

Rather than search for results in the general setting of Section \ref{sec:preliminaries}, we decided, for the sake of clarity, to look at pinwheel solutions only. The solutions found in \cite{pw,pv} for $N=2,3$ and $p=2$ were of this type. Peng and Wang \cite{pw} focused on the case where the potential $V$ is greater than its limit at infinity and, for a system of two equations, they established the existence of pinwheel solutions for $\beta$ sufficiently negative. Pistoia and Vaira raised the question of whether solutions exist when $V$ is below its limit at infinity and showed in \cite{pv} that the system \eqref{eq:system} has a solution satisfying \eqref{eq:symmetries} for $\beta$ close enough to $0$. The energy of each component approaches $n\mathfrak{c}_\infty$ as $\beta\to 0$. 

Our results can be easily extended to dimension $N=2$. In contrast, the dimension $N=3$ requires a more delicate analysis because compactness can also be lost by the presence of solutions to the autonomous system (with $V=V_\infty$) that travel to infinity; see Remark \ref{rem:N=3}.

The paper is organized as follows. In Section \ref{sec:preliminaries} we present the general variational framework and in Section \ref{sec3} we study the behavior of minimizing sequences of pinwheel solutions for the system \eqref{eq:system}. In Section \ref{sec4} we prove Theorem \ref{thm:existence} and Proposition \ref{prop:symmetry_breaking}. Section \ref{sec5} is devoted to the proofs of Theorems \ref{thm:beta_to_0} and \ref{thm:partition}.

\section{The symmetric variational setting}
\label{sec:preliminaries}

Let $G$ be a closed subgroup of the group $O(N)$ of linear isometries of $\rn$, and for $\ell\geq 2$ let $S_\ell$ be the group of permutations of the set $\{1,\ldots,\ell\}$ acting on $\r^\ell$ in the obvious way, i.e.,
$$\sigma(u_1,\ldots,u_\ell)=(u_{\sigma(1)},\ldots,u_{\sigma(\ell)})\text{ \ for every \ }\sigma\in S_\ell, \ (u_1,\ldots,u_\ell)\in\r^\ell.$$
Let $\phi:G\to S_\ell$ be a continuous homomorphism of groups.
A function $\bf u:\rn\to\r^\ell$ will be called \emph{$\phi$-equivariant} if
\begin{equation}\label{eq:equivariant}
\bf u(gx)=\phi(g)\bf u(x) \text{ \ for all \ }g\in G, \ x\in\rn.
\end{equation}
Note that, if $\bf u:\rn\to\r^\ell$ is $\phi$-equivariant, then $\bf u$ is $K_\phi$-invariant, where $K_\phi:=\ker(\phi)$.

These data define a $G$-action on $\cH:=(H^1(\rn))^\ell$ as follows:
$$(g\bf u)(x):=\phi(g)\bf u(g^{-1}x)\text{ \ for every \ }g\in G, \ \bf u=(u_1,\ldots,u_\ell)\in\cH.$$
For $u,v\in H_0^1(\rn)$  we set 
$$\langle u,v\rangle_V:=\irn(\nabla u\cdot\nabla v+V(x)uv)\qquad\text{and}\qquad\|u\|_V:=\sqrt{\langle u,u\rangle_V}.$$
The solutions to the system \eqref{eq:system} are the positive critical points of the functional $\cJ:\cH\to\r$ given by 
$$\cJ(\bf u):= \frac{1}{2}\sum_{i=1}^\ell\|u_i\|_V^2 - \frac{1}{2p}\sum_{i=1}^\ell\irn |u_i|^{2p}-\frac{\beta}{2p}\sum_{\substack{i,j=1 \\ i\neq j}}^\ell\irn |u_i|^p|u_j|^p,$$
which is of class $\cC^1$. Its $i\text{-th}$ partial derivative is
$$\partial_i\cJ(\bf u)v=\langle u_i,v\rangle_V-\irn|u_i|^{2p-2}u_iv-\beta\sum\limits_{\substack{j=1\\j\neq i}}^\ell\irn|u_j|^p|u_i|^{p-2}u_iv$$
for any $\bf u\in\cH, \ v\in H^1(\rn)$. The functional $\cJ$ is $G$-invariant, i.e.,
$$\cJ(g\bf u)=\cJ(\bf u)\quad\text{ \ for every \ }g\in G, \ \bf u=(u_1,\ldots,u_\ell)\in\cH.$$
So, by the principle of symmetric criticality \cite[Theorem 1.28]{w}, the critical points of the restriction of $\cJ$ to the $G$-fixed point space of $\cH$,
$$\cH^\phi:=\{\bf u\in\cH:g\bf u=\bf u \ \forall g\in G\}=\{\bf u\in\cH:\bf u\text{ is }\phi\text{-equivariant}\},$$
are critical points of $\cJ$, i.e., they are the solutions to the system \eqref{eq:system} satisfying \eqref{eq:equivariant}. We denote by $\cJ^\phi$ the restriction of $\cJ$ to $\cH^\phi$. Note that
\begin{align*}
(\cJ^\phi)'(\bf u)\bf v=\cJ'(\bf u)\bf v=\sum_{i=1}^\ell\partial_i\cJ(\bf u)v_i\qquad\text{for any \ }\bf u,\bf v\in\cH^\phi.
\end{align*}
The fully nontrivial critical points of $\cJ^\phi$ belong to the set
\begin{align*}
\cN^\phi:=\{\bf u\in\cH^\phi:u_i\neq 0, \ \partial_{i}\cJ(\bf u)u_i=0 \ \forall i=1,\ldots,\ell\}.
\end{align*}
Observe that
$$\cJ^\phi(\bf u)= \frac{p-1}{2p}\sum_{i=1}^\ell\|u_i\|_V^2\qquad\text{if \ }\bf u\in\cN^\phi.$$
Set
\begin{equation*}
c^\phi:= \inf_{\bf u\in \cN^\phi}\cJ^\phi(\bf u).
\end{equation*}

We consider also the single equation
\begin{equation} \label{eq:G-equation}
-\Delta u + V(x)u = |u|^{2p-2}u,\qquad u\in H^1(\rn)^{G},
\end{equation}
where \ $H^1(\rn)^G:=\{u\in H^1(\rn):u\text{ is }G\text{-invariant}\}$, and we denote by $J:H^1(\rn)^{G}\to\r$ and $\cM^{G}$ energy functional and the Nehari manifold associated to it, i.e.,
\begin{equation}\label{eq:J}
J(u):=\frac{1}{2}\|u\|_V^2 - \frac{1}{2p}\irn |u|^{2p}
\end{equation}
and
$$\cM^G:=\left\{u\in H^1(\rn)^G:u\neq 0, \ \|u\|_V^2 =\irn |u|^{2p}\right\}.$$
Similarly, we denote by $J_\infty:H^1(\rn)\to\r$ and $\cM_\infty$ be the energy functional and the Nehari manifold associated to \eqref{eq:limit_problem}. Set 
\begin{equation}\label{eq:frak_c}
\mathfrak{c}_\infty:=\inf_{u\in\cM_\infty}J_\infty(u)\qquad\text{and}\qquad\mathfrak{c}^{G}:=\inf_{u\in\cM^{G}}J(u).
\end{equation}

We shall focus our attention on the following example.

\begin{example} \label{example}
Let $\z_{m}:=\{\mathrm{e}^{2\pi\mathrm{i}j/m}:j=0,\ldots,m-1\}$ act on $\cc$ by complex multiplication and $G_m:=\z_m\times O(N-2)$ act on $\rn$ as
\begin{align*}
\alpha x:=&(\alpha z,y)&&\forall \alpha\in\z_m, \\
\theta x:=&(z,\theta y)&&\forall \theta\in O(N-2), \quad x=(z,y)\in\cc\times\r^{N-2}\equiv\rn.
\end{align*}
Let $\sigma_1\in S_\ell$ be the cyclic permutation \ $\sigma_1(i):=i+1 \mod\ell$, and $\phi_n:G_{\ell n}\to S_\ell$ be the homomorphism given by \ $\phi_n(\mathrm{e}^{2\pi\mathrm{i}/\ell n},\theta):=\sigma_1$ \ for any $\theta\in O(N-2)$. Then $\bf u:\rn\to\r^\ell$ is $\phi_n$-equivariant iff 
$$(u_1(\mathrm{e}^{2\pi\mathrm{i}/\ell n}z,\theta y),\ldots,u_\ell(\mathrm{e}^{2\pi\mathrm{i}/\ell n}z,\theta y))=(u_2(z,y),\ldots,u_\ell(z,y),u_1(z,y))$$
for every $(z,y)\in\cc\times\r^{N-2}$ and $\theta\in O(N-2)$, i.e., iff \eqref{eq:symmetries} hold true. Note that every $u_j$ is $G_n$-invariant.
\end{example}

\section{The behavior of minimizing sequences}
\label{sec3}

From now on we fix $n$ and we take $G_n$ and $\phi_n:G_{\ell n}\to S_\ell$ as in Example \ref{example}. Then, for any $\bf u,\bf v\in\cH^{\phi_n}$,
\begin{align} \label{eq:partial}
(\cJ^{\phi_n})'(\bf u)\bf v=\sum_{i=1}^\ell\partial_i\cJ(\bf u)v_i=\ell\,\partial_j\cJ(\bf u)v_j\qquad\text{for any \ }j=1,\ldots,\ell,
\end{align}
and the set $\cN^{\phi_n}$ is the usual Nehari manifold associated to the functional $\cJ^{\phi_n}:\cH^{\phi_n}\to\r$, i.e.,
\begin{align*}
\cN^{\phi_n}=\{\bf u\in\cH^{\phi_n}:\bf u\neq 0, \ (\cJ^{\phi_n})'(\bf u)\bf u=0\}.
\end{align*}
It has the following properties.

\begin{proposition} \label{prop:nehari}
\begin{itemize}
\item[$(a)$] $\cN^{\phi_n}\neq\emptyset$. 
\item[$(b)$] $c^{\phi_n}\geq\ell\mathfrak{c}^{G_n}>0$.
\item[$(c)$] $\cN^{\phi_n}$ is a closed $\cC^1$-submanifold of codimension $1$ of $\cH^{\phi_n}$, and a natural constraint for $\cJ^{\phi_n}$.
\item[$(d)$] If $\bf u\in\cH^{\phi_n}$ is such that, for each $i=1,\ldots,\ell$,
\begin{equation*}
\irn|u_i|^{2p} + \sum_{\substack{j=1 \\ j\neq i}}^\ell\beta\irn|u_i|^p|u_j|^p>0,
\end{equation*}
then there exists a unique $s_{\bf u} \in (0,\infty)$ such that \ $s_{\bf u}\bf u\in \cN^{\phi_n}$. Furthermore,  
$$\cJ^{\phi_n}(s_{\bf u}\bf u) = \max_{s\in (0,\infty)}\cJ^{\phi_n}(s \bf u).$$
\item[$(e)$] $c^{\phi_n}\leq \ell n\mathfrak{c_\infty}$.
\end{itemize}
\end{proposition}

\begin{proof}
The proof is easy. We give the details for the sake of completeness.

$(a):$ \ Let $\vp\in\cC_c^\infty(\rn)$ be a nontrivial radial function such that $\|\vp\|_V^2=\irn|\vp|^{2p}$. Set $\xi_{i,j}:=(\e^{2\pi\mathrm{i}(i+\ell j)/\ell n},0)\in\cc\times\r^{N-2}\equiv\rn$ and define
\begin{align*}
u_{R,i+1}(x):=\sum_{j=0}^{n-1}\vp(x-R\xi_{i,j}),\qquad i=0,\ldots,n-1,
\end{align*}
where $R>0$ is taken large enough so that $u_{R,i}$ and $u_{R,j}$ have disjoint supports for every $i\neq j$. Then, $\bf u_R:=(u_{R,1},\ldots,u_{R,1})\in \cN^{\phi_n}$.

$(b):$ \ Let $\bf u=(u_1,\ldots,u_\ell)\in\cN^{\phi_n}$. As $\beta<0$ we have
$$0<\|u_i\|_V^2=\|u_1\|_V^2 \leq \irn|u_1|^{2p}=\irn|u_i|^{2p}\qquad\forall i=2,\ldots,\ell.$$
Hence, there exists $s\in(0,1]$ such that $su_i\in\cM^{G_n}$ for every $i=1,\ldots,\ell$. Therefore,
$$\ell\mathfrak{c}^{G_n}\leq\sum_{i=1}^\ell J(s_iu_i)=\frac{p-1}{2p}\sum_{i=1}^\ell\|s_iu_i\|^2_V\leq\frac{p-1}{2p}\sum_{i=1}^\ell\|u_i\|^2_V=\cJ^\phi(\bf u).$$
It follows that \ $\ell\mathfrak{c}^{G_n}\leq c^{\phi_n}$.

$(c):$ \ The function $\Psi:\cH^{\phi_n}\smallsetminus\{0\}\to\r$ given by $\Psi(\bf u):=(\cJ^{\phi_n})'(\bf u)\bf u$ is of class $\cC^1$ and $\cN^{\phi_n}=\Psi^{-1}(0)$. It follows from $(b)$ that $\cN^{\phi_n}$ is a closed subset of $\cH^{\phi_n}$. As
$$\Psi'(\bf u)\bf u=(2-2p)\ell\,\|u_1\|_V^2\neq 0,$$
we have that $0$ is a regular value of $\Psi$. This shows that $\cN^{\phi_n}$ is a $\cC^1$-submanifold of codimension $1$ of $\cH^{\phi_n}$. It also shows that $\bf u\not\in\ker\Psi'(\bf u)=:T_{\bf u}\cN^{\phi_n}$, the tangent space of $\cN^{\phi_n}$ at $\bf u$. Hence,
$$\cH^{\phi_n}=T_{\bf u}\cN^{\phi_n}\oplus\r\bf u.$$
Since, by definition, $(\cJ^{\phi_n})'(\bf u)\bf u=0$ for every $\bf u\in\cN^{\phi_n}$, we infer that a critical point of the restriction of  $\cJ^{\phi_n}$ to $\cN^{\phi_n}$ is a critical point of $\cJ^{\phi_n}$. 

$(d):$ \ The proof is straightforward. The number $s_{\bf u}$ is 
$$s_{\bf u}=\left(\frac{\|u_1\|_V^2}{\irn|u_1|^{2p} + \sum\limits_{\substack{j=1 \\ j\neq 1}}^\ell\beta\irn|u_1|^p|u_j|^p}\right)^{1/(2p-2)}.$$

$(e):$ \ Let $\omega$ be the least energy positive radial solution to \eqref{eq:limit_problem} and set $\xi_{i,j}=(\e^{2\pi\mathrm{i}(i+\ell j)/\ell n},0)\in\cc\times\r^{N-2}\equiv\rn$. Define
\begin{align*}
w_{R,i+1}(x)&:=\sum_{j=0}^{n-1}\omega(x-R\xi_{i,j}),\qquad i=0,\ldots,n-1.
\end{align*}
Then $\bf w_R=(w_{R,1},\ldots,w_{R,\ell})\in\cH^{\phi_n}$. If $R$ is sufficiently large, statement $(d)$ yields $s_R\in(0,\infty)$ such that $s_R\bf w_R\in\cN^{\phi_n}$ and $s_{R}\to 1$ as $R\to\infty$. Using assumption $(V_2)$ we obtain
\begin{align*}
c^{\phi_n}&\leq\cJ^{\phi_n}(s_R\bf w_R)=\frac{p-1}{2p}\sum_{i=1}^{\ell}\|s_{R,i}w_{R,i}\|_V^2\to\ell n\mathfrak{c}_ \infty\quad\text{as \ }R\to\infty.
\end{align*}
This shows that $c^{\phi_n}\leq \ell n\mathfrak{c_\infty}$, as claimed.
\end{proof}

\begin{lemma}\label{lem:orbits}
Let $(x_k)$ be a sequence in $\rn$, $N\geq 4$. After passing to a subsequence, there exists a sequence $(\xi_k)$ in $\rn$ and a constant $C_0 >0$ such that
$$|x_k-\xi_k| \leq C_0 \qquad\text{for all \ }k\in\n,$$
and one of the following statements holds true:
\begin{itemize}
\item either $\xi_k =0$ for all $k$,
\item or $\xi_k = (\zeta_k,0)\in\cc\times\r^{N-2}$ and $|\zeta_k|\to\infty$,
\item or for each $m\in\n$ there exist $\gamma_1,\ldots,\gamma_m \in O(N-2)$ such that $|\gamma_i \xi_k - \gamma_j \xi_k| \to \infty$ if $i \neq j$.
\end{itemize}
\end{lemma}

\begin{proof}
See \cite[Lemma 3.1]{cmp}.
\end{proof}

\begin{remark} \label{rem:N=3}
\emph{Note that this lemma is not true in dimension $N=3$ because $O(1)=\{1,-1\}$.}
\end{remark}

\begin{theorem} \label{thm:splitting} 
Let $\bf u_k=(u_{k,1},\ldots,u_{k,\ell})\in\cN^{\phi_n}$ be such that $\cJ^{\phi_n}(\bf u_k)\to c^{\phi_n}$ and $u_{k,i}\geq 0$. Then, after passing to a subsequence, either $\bf u_k\to\bf u$ strongly in $\cH^{\phi_n}$ with $u_i\geq 0$, or there are points $(z_{k},0)\in\cc\times\r^{N-2}\equiv\rn$ such that $|z_{k}|\to\infty$,
$$\lim_{k\to\infty}\Big\|u_{k,1}-\sum_{j=0}^n\omega\big( \ \cdot \ -\,(\mathrm{e}^{2\pi\mathrm{i}j/n}z_{k},0)\big)\Big\|=0,$$
and \ $c^{\phi_n} = \ell n\mathfrak{c}_\infty$, where $\omega$ is the least energy positive radial solution to \eqref{eq:limit_problem}. 
\end{theorem}

\begin{proof}
Invoking Ekeland's variational principle \cite[Theorem 8.5]{w} we may assume that $(\cJ^{\phi_n})'(\bf u_k)\to 0$ in $(\cH^{\phi_n})'$. 

Since $\beta<0$, Proposition \ref{prop:nehari}$(b)$ yields $c_0>0$ such that
$$\irn|u_{k,1}|^{2p}>c_0\qquad\forall k\in\n.$$
By Lions' lemma \cite[Lemma 1.21]{w} there exist $\delta>0$ and $x_k\in\rn$ such that, after passing to a subsequence,
\begin{equation*}
\int_{B_1(x_k)}|u_{k,1}|^{2p}>\delta\qquad\forall k\in\n.
\end{equation*}
For $(x_k)$ we fix a sequence $(\xi_k)$ and a constant $C_0>0$ such that $|x_k-\xi_k|\leq C_0$ for all $k\in\n$, satisfying one of the alternatives stated in Lemma \ref{lem:orbits}. Then,
\begin{equation} \label{eq:nonnull}
\int_{B_{C_0+1}(\xi_k)}|u_{k,1}|^{2p}\geq\int_{B_1(x_k)}|u_{k,1}|^{2p}>\delta\qquad\forall k\in\n.
\end{equation}
It follows that, either $\xi_k=0$, or $\xi_k = (\zeta_k,0) \in \mathbb{C} \times \mathbb{R}^{N-2}$ and $\zeta_k\to\infty$. Otherwise, by Lemma \ref{lem:orbits}, for each $m\in\n$ there would exist $\gamma_1,\ldots,\gamma_m\in O(N-2)$ such that $|\gamma_i\xi_k-\gamma_j\xi_k|\geq 2(C_0+1)$ if $i\neq j$ for large enough $k\in \n$ and, as $u_{k,1}$ is $G_n$-invariant, we would have that
\begin{align*}
\irn|u_{k,1}|^{2p}\geq\sum_{i=1}^m\int_{B_{C_0+1}(\gamma_i\xi_k)}|u_{k,1}|^{2p}=m\int_{B_{C_0+1}(\xi_k)}|u_{k,1}|^{2p}>m\delta,
\end{align*}
for all $m\in\n$. This is impossible because $(u_{k,1})$ is bounded in $L^{2p}(\rn)$.

Next, we distinguish two cases.
\smallskip  

\underline{Case 1.} \ $\xi_k=0$ for all $k\in\n$.

Since the sequence $(u_{k,1})$ is bounded in $H^1(\rn)$, passing to a subsequence, we have that $u_{k,1}\rh u_1$ weakly in $H^1(\rn)$, $u_{k,1}\to u_1$ in $L^{2p}_\mathrm{loc}(\rn)$ and $u_{k,1}\to u_1$ a.e. in $\rn$. Hence, $u_1\geq 0$ and it follows from \eqref{eq:nonnull} that $u_1\neq 0$. Note that, as $u_{k,1}\in H^1(\rn)^{G_n}$, $u_1\in H^1(\rn)^{G_n}$. Set $u_{j+1}(z,y):=u_1(\e^{2\pi\i j/\ell n}z,y)$ for $(z,y)\in\cc\times\r^{N-2}$, $j=1,\ldots,\ell-1$, and $\bf u=(u_1,\ldots,u_\ell)$. Then, $u_{k,j+1}\rh u_{j+1}$ weakly in $H^1(\rn)$ and, as $(\cJ^{\phi_n})'(\bf u_k)\to 0$ in $(\cH^{\phi_n})'$, we derive from \eqref{eq:partial} that
\begin{align*}
0=\lim_{k\to\infty}\partial_1\cJ(\bf u_k)\vp=\partial_1\cJ(\bf u)\vp\qquad\text{for every \ }\vp\in\cC^\infty_c(\rn)^{G_n}.
\end{align*}
Hence, $\bf u\in\cN^{\phi_n}$ and
\begin{align*}
c^{\phi_n}&\leq\cJ^{\phi_n}(\bf u)=\frac{p-1}{2p}\sum_{i=1}^\ell\|u_i\|_V^2\leq\liminf_{k\to\infty}\frac{p-1}{2p}\sum_{i=1}^\ell\|u_{k,i}\|_V^2 \\
&=\lim_{k\to\infty}\cJ^{\phi_n}(\bf u_k)= c^{\phi_n}.
\end{align*}
Therefore, $\bf u_k\to\bf u$ strongly in $\cH^{\phi_n}$. This shows that, in Case 1, the first  alternative stated in Theorem \ref{thm:splitting} holds true.
\smallskip

\underline{Case 2.} \ $\xi_k = (\zeta_k,0) \in \mathbb{C} \times \mathbb{R}^{N-2}$ and $\zeta_k\to\infty$.

Set
$$w_{k,i}(x):=u_{k,i}(x+\xi_{k}),\qquad i=1,\ldots,\ell.$$
Note that $w_{k,i}$ is $O(N-2)$-invariant. Since the sequence $(w_{k,i})$ is bounded in $H^1(\rn)$, a subsequence satisfies $w_{k,i}\rh w_i$ weakly in $H^1(\rn)^{O(N-2)}$, $w_{k,i}\to w_i$ in $L^{2p}_\mathrm{loc}(\rn)$ and $w_{k,i}\to w_i$ a.e. in $\rn$. Hence, $w_i\geq 0$. To simplify notation, set $\alpha:=\e^{2\pi\i /n}$. Note that, as $|\alpha^j\xi_k-\alpha^m\xi_k|\to\infty$ if $j\neq m$, we have that
$$w_{k,i}\circ\alpha^{-m}-\sum_{j=m+1}^{n-1}(w_{i}\circ\alpha^{-m})(\,\cdot\,-\alpha^j\xi_k+\alpha^m\xi_k)\rh w_i\circ\alpha^{-m}$$
weakly in $H^1(\rn)$. Hence, setting $V_k(x):=V(x+\xi_k)$, Lemma \ref{lem:A1} gives
\begin{align*}
\|w_{i}\circ\alpha^{-m}\|_{V_\infty}^2 &
=\Big\|w_{k,i}\circ\alpha^{-m}-\sum_{j=m+1}^{n-1}(w_{i}\circ\alpha^{-j})(\,\cdot\,-\alpha^j\xi_k+\alpha^m\xi_k)\Big\|_{V_k}^2\\
&\quad-\Big\|w_{k,i}\circ\alpha^{-m}-\sum_{j=m}^{n-1}(w_{i}\circ\alpha^{-j})(\,\cdot\,-\alpha^j\xi_k+\alpha^m\xi_k)\Big\|_{V_k}^2+o(1).
\end{align*}
Since $u_{k,i}$ is $G_n$-invariant, the change of variable $y=z-\alpha^m\xi_k$ yields
\begin{align*}
&\Big\|u_{k,i}-\sum_{j=m+1}^{n-1}(w_{i}\circ\alpha^{-j})(\,\cdot\,-\alpha^j\xi_k)\Big\|_V^2\\
&\qquad=\Big\|u_{k,i}-\sum_{j=m}^{n-1}(w_{i}\circ\alpha^{-j})(\,\cdot\,-\alpha^j\xi_k)\Big\|_V^2+\|w_{i}\|_{V_\infty}^2+o(1),
\end{align*}
and iterating this identity we obtain
\begin{equation} \label{eq:norms}
\|u_{k,i}\|_V^2=\Big\|u_{k,i}-\sum_{j=0}^{n-1}(w_{i}\circ\alpha^{-j })(\,\cdot\,-\alpha^j\xi_k)\Big\|_V^2+n\|w_{i}\|_{V_\infty}^2+o(1).
\end{equation}
On the other hand, for any given $v\in H^1(\rn)^{O(N-2)}$ set $v_{k}(y):=v(y-\xi_{k})$ and
$$\what v_k(y):=\sum_{j=0}^{n-1}v_k(\alpha^jy).$$
Recalling that $u_{k,i}$ is $G_n$-invariant and performing the translation $y=x+\xi_k$, we obtain
\begin{align*}
\partial_i\cJ(\bf u_k)\what v_k &=\sum_{j=0}^{n-1}\partial_i\cJ(\bf u_k)(v_k\circ \alpha^j)=n\,\partial_i\cJ(\bf u_k)v_k\\
&=n\Big(\irn(\nabla w_{k,i}\cdot\nabla v+V_k(x)w_{k,i}v)-\irn|w_{k,i}|^{2p-2}w_{k,i}v \nonumber \\
&\qquad-\beta\sum\limits_{\substack{j=1\\j\neq i}}^\ell\irn|w_{k,j}|^p|w_{k,i}|^{p-2}w_{k,i}v\Big). 
\end{align*}
Note that $\what v_k$ is $G_n$-invariant. As $(\cJ^{\phi_n})'(\bf u_k)\to 0$, invoking \eqref{eq:partial} and assumption $(V_2)$, and passing to the limit as $k\to\infty$ we get
\begin{equation} \label{eq:solution}
0=\irn(\nabla w_i\cdot\nabla v+V_\infty w_iv)-\irn|w_i|^{2p-2}w_iv-\beta\sum\limits_{\substack{j=1\\j\neq i}}^\ell\irn|w_j|^p|w_i|^{p-2}w_iv
\end{equation}
for every $v\in H^1(\rn)^{O(N-2)}$ and $i=1,\ldots,\ell$. Since, by \eqref{eq:nonnull},
\begin{equation*}
\int_{B_{C_0+1}(0)}|w_{k,1}|^{2p}\geq\int_{B_{C_0+1}(\xi_k)}|u_{k,1}|^{2p}\geq\delta>0,
\end{equation*}
we see that $w_1\neq 0$. Furthermore, equation \eqref{eq:solution} implies that
\begin{equation} \label{eq:t_i}
\|w_1\|_{V_\infty}^2=\irn|w_1|^{2p}+\beta\sum\limits_{\substack{j=1\\j\neq 1}}^\ell\irn|w_j|^p|w_1|^{p}\leq\irn|w_1|^{2p},
\end{equation}
so there exists $t\in(0,1]$ such that $\|tw_1\|_{V_\infty}^2=\irn|tw_1|^{2p}$. It follows that $tw_1\in\cM_\infty$, and from equation \eqref{eq:norms} and Proposition \ref{prop:nehari}$(e)$ we derive
\begin{align*}
n\mathfrak{c}_\infty\leq\frac{p-1}{2p}n\|tw_1\|_{V_\infty}^2\leq\frac{p-1}{2p}n\|w_1\|_{V_\infty}^2\leq\lim_{k\to\infty}\frac{p-1}{2p}\|u_{k,1}\|_V^2=\frac{1}{\ell}c^{\phi_n}\leq n\mathfrak{c}_\infty.
\end{align*}
Therefore, $t=1$, $w_1\in\cM_\infty$ and $J_\infty(w_1)=\frac{p-1}{2p}\|w_1\|_{V_\infty}^2=\mathfrak{c}_\infty$, i.e., $w_1$ is a least energy solution of \eqref{eq:limit_problem}. Moreover, from \eqref{eq:norms} we get that
$$\lim_{k\to\infty}\Big\|u_{k,1}-\sum_{j=0}^{n-1}(w_1\circ\alpha^{-j })(\,\cdot\,-\alpha^j\xi_k)\Big\|_V^2=0.$$
Since the positive least energy solution to \eqref{eq:limit_problem} is unique up to translation and $w_1$ is $O(N-2)$-invariant, there exists $\xi=(\zeta,0)\in\cc\times\r^{N-2}$ such that $w_1(x)=\omega(x+\xi)$. Hence, $(w_1\circ\alpha^{-j })(x-\alpha^j\xi_k)=\omega(\alpha^{-j}x-\xi_k-\xi)=\omega(x-\alpha^j(\xi_k+\xi))$. So, setting $z_k:=\zeta_k+\zeta$, we obtain
$$\lim_{k\to\infty}\Big\|u_{k,1}-\sum_{j=0}^n\omega\big( \ \cdot \ -\,(\mathrm{e}^{2\pi\mathrm{i}j/n}z_{k},0)\big)\Big\|=0.$$
This shows that, in Case 2, the second alternative stated in Theorem \ref{thm:splitting} holds true.
\end{proof}

\begin{corollary}\label{cor:compactness}
If $c^{\phi_n} < \ell n\mathfrak{c}_\infty$, the system \eqref{eq:system} has a least energy fully nontrivial solution satisfying \eqref{eq:symmetries}.
\end{corollary}
 
\section{Existence of a solution} \label{sec4}

We define the set of \emph{weak $(n,\ell)$-pinwheel partitions} as
\begin{align*}
\cW^n_\ell:=\{(u_1,\ldots,u_\ell)\in\cH^{\phi_n}:u_i\neq 0,\,\|u_i\|_V^2=|u_i|_{2p}^{2p},\,u_iu_j=0\text{ in }\rn\text{ if }i\neq j\},
\end{align*}
and set
$$\widehat{c}^{\phi_n}:=\inf_{(u_1,\ldots,u_\ell)\in\cW^n_\ell}\,\frac{p-1}{2p}\sum_{i=1}^\ell\|u_i\|_V^2.$$
Our next goal is to give an upper estimate for $\widehat{c}^{\phi_n}$. To this end, we choose $\eps\in\big(0,\frac{d_{\ell n}-\lambda}{d_{\ell n}+\lambda}\big)$ and a radial function $\chi\in\cC^\infty(\rn)$ satisfying $0\leq\chi\leq 1$, $\chi(x)=1$ if $|x|\leq 1-\eps$ and $\chi(x)=0$ if $|x|\geq 1$. Let $\omega$ be the positive least energy radial solution to \eqref{eq:limit_problem}. For each $r>0$ define
$$\omega_r(x):=\chi\left(\frac{x}{r}\right)\omega(x).$$

\begin{lemma} \label{lem:estimates_omega}
As $r\to\infty$,
$$\left| \|\omega\|^2-\|\omega_r\|^2 \right|=O(\e^{-2(1-\eps)\sqrt{V_\infty}r}),\qquad\left| |\omega|^{2p}_{2p}-|\omega_r|^{2p}_{2p} \right|=O(\e^{-2p(1-\eps)\sqrt{V_\infty}r}),$$
where $|\,\cdot\,|_{2p}$ denotes the norm in $L^{2p}(\rn)$.
\end{lemma}

\begin{proof}
This statements follow easily from the well known estimates $|\omega(x)|=O(|x|^{-\frac{N-1}{2}}\e^{-\sqrt{V_\infty}|x|})$ and $|\nabla\omega(x)|=O(|x|^{-\frac{N-1}{2}}\e^{-\sqrt{V_\infty}|x|})$, as in \cite[Lemma 2]{cw}.
\end{proof}

Set $\varrho:=\frac{d_{\ell n}+\lambda}{4}$, and for $R>1$ define
$$\widehat w_{1,R}(x):=\sum_{j=0}^{n-1}\omega_{\varrho R}(x-R(\e^{2\pi\i j/n},0))\qquad\text{and}\qquad w_{1,R}:=t_R\widehat w_{1,R},$$
where $t_R\in(0,\infty)$ is such that $\|w_{1,R}\|_V^2=|w_{1,R}|_{2p}^{2p}$. Note that $t_R\to 1$ as $R\to\infty$, \ $w_{1,R}$ is $G_n$-invariant and 
$$\supp\big(\omega_{\varrho R}(\,\cdot\,-R(\e^{2\pi\i j/\ell n},0))\big)\subset\overline{ B_{\varrho R}(R(\e^{2\pi\i j/\ell n},0))}.$$
Set $w_{j+1,R}(\e^{2\pi\i j/\ell n}z,y):=w_{1,R}(z,y)$ for $(z,y)\in\cc\times\r^{N-2}$ and $j=1,\ldots,\ell-1$. Since $\varrho<\frac{d_{\ell n}}{2}$ we have that $\supp(w_{i,R})\cap\supp(w_{j,R})=\emptyset$ if $i\neq j$. Hence, $\bf w_R=(w_{1,R},\ldots,w_{\ell,R})\in\cW^n_\ell$.

\begin{lemma} \label{lem:energy_estimate}
There exist $C_1,R_1>0$ such that
$$\frac{p-1}{2p}\sum_{i=1}^\ell\|w_{i,R}\|_V^2=\cJ^{\phi_n}(\bf w_R)\leq \ell n\mathfrak{c}_\infty-C_1\e^{-\lambda\sqrt{V_\infty} R}\quad\text{for all \ }R\geq R_1.$$
\end{lemma}

\begin{proof}
Since $\bf w_R=(w_{1,R},\ldots,w_{\ell,R})\in\cW^n_\ell$, the equality holds true. To prove the inequality note that that $t_R\in[\frac{1}{2},2]$ for $R$ large enough. Assumption $(V_3^n)$ yields
\begin{align*}
&\irn(V(x)-V_\infty)|t_R\omega_{\varrho R}(x-R(1,0))|^2\d x\\
&\qquad=\int_{|x|\leq\varrho R}\big(V(x+R(1,0))-V_\infty\big)|t_R\omega_{\varrho R}(x)|^2\d x\\
&\qquad=-\frac{C_0}{4}\int_{|x|\leq\varrho R}\e^{-\lambda\sqrt{V_\infty}\,|x+R(1,0)|}|\omega(x)|^2\d x\\
&\qquad\leq -\frac{C_0}{4}\Big(\irn\e^{-\lambda\sqrt{V_\infty}|x|}|\omega(x)|^2\d x\Big)\e^{-\lambda \sqrt{V_\infty}R}=:-2C\e^{-\lambda\sqrt{V_\infty} R}.
\end{align*}
Using Lemma \ref{lem:estimates_omega}, for $R$ large enough we get
\begin{align*}
&\cJ^{\phi_n}(\bf w_R)= \frac{1}{2}\sum_{i=1}^\ell\|w_{i,R}\|_V^2 - \frac{1}{2p}\sum_{i=1}^\ell\irn |w_{i,R}|^{2p}-\frac{\beta}{2p}\sum_{\substack{i,j=1 \\ i\neq j}}^\ell\irn |w_{i,R}|^p|w_{j,R}|^p \\
&=\ell n\left(\frac{1}{2}\|t_R\omega_{\varrho R}(\,\cdot\,-R(1,0)) \|_V^2 - \frac{1}{2p}|t_R\omega_{\varrho R}(\,\cdot\,-R(1,0))|^{2p}_{2p}\right)\\
&=\ell n\left(\frac{1}{2}\|t_R\omega_{\varrho R}\|_{V_\infty}^2 + \frac{1}{2}\irn(V-V_\infty)|t_R\omega_{\varrho R}(\,\cdot\,-R(1,0))|^2 - \frac{1}{2p}|t_R\omega_{\varrho R}|^{2p}_{2p}\right)\\
&=\ell n\left(\frac{1}{2}\|t_R\omega\|_{V_\infty}^2 - C\e^{-\lambda\sqrt{V_\infty} R} - \frac{1}{2p}|t_R\omega|^{2p}_{2p} +\, O\big(\e^{-2(1-\eps)\sqrt{V_\infty}\varrho R}\big)\right)\\
&\leq \ell n\mathfrak{c}_\infty - C_1\e^{-\lambda\sqrt{V_\infty} R},
\end{align*}
because $2(1-\eps)\varrho>\frac{d_{\ell n}+\lambda}{2}\big(1-\frac{d_{\ell n}-\lambda}{d_{\ell n}+\lambda}\big)=\lambda$.
\end{proof}
\smallskip

\begin{proof}[Proof of Theorem \ref{thm:existence}]
Note that $\cW^n_\ell\subset\cN^{\phi_n}$. Hence, from Lemma \ref{lem:energy_estimate} we get
$$c^{\phi_n}\leq\widehat{c}^{\phi_n}<\ell n\mathfrak{c}_\infty,$$
and Corollary \ref{cor:compactness} yields the result.
\end{proof}
\smallskip

\begin{proof}[Proof of Proposition \ref{prop:symmetry_breaking}]
Arguing by contradiction, assume that $\bf u$ is a solution to \eqref{eq:system} satisfying \eqref{eq:symmetries} with $n=\ell^m$ and with $n=\ell^q$ and that $1\leq m<q$. Then, for $k=\ell^{q-m-1}j$ with $j=1,\ldots,\ell-1$ we have that
$$u_1(x)=u_1(\mathrm{e}^{2\pi\mathrm{i}k/\ell^q}x)=u_1(\mathrm{e}^{2\pi\mathrm{i}j/\ell^m\ell}x)=u_{j+1}(x)$$
and, as $1+\beta(\ell-1)\leq 0$, we obtain
\begin{align*}
\|u_1\|_V^2 =\irn|u_1|^{2p}+\beta\sum_{j=1}^{\ell-1}|u_{j+1}|^p|u_1|^p=(1+\beta(\ell-1))\irn|u_1|^{2p}\leq 0,
\end{align*}
a contradiction.
\end{proof}

\section{The limit profiles of the solutions}\label{sec5}

We start with the case $\beta\to 0$.

\begin{proof}[Proof of Theorem \ref{thm:beta_to_0}]
We write $\cJ_k^{\phi_n}$ and $\cN_k^{\phi_n}$ for the functional and the Nehari set associated to the system \eqref{eq:system} with $\beta=\beta_k$, and we define
$$c_k^{\phi_n}:=\inf_{\cN_k^{\phi_n}}\cJ_k^{\phi_n}.$$
As $\cW^n_\ell\subset\cN_k^{\phi_n}$ for every $k\in\n$, invoking Lemma \ref{lem:energy_estimate} we see that
\begin{equation}\label{eq:estimate}
\frac{p-1}{2p}\sum_{i=1}^\ell\|u_{k,i}\|_V^2=c_k^{\phi_n}\leq \widehat{c}^{\phi_n}<\ell n\mathfrak{c}_\infty\qquad\forall k\in\n.
\end{equation}
After passing to a subsequence, we have that $u_{k,i} \rh u_{0,i}$ weakly in $H^1(\rn)$, $u_{k,i} \to u_{0,i}$ strongly in $L^2_\mathrm{loc}(\rn)$ and $u_{k,i} \to u_{0,i}$ a.e. in $\rn$, for each $i=1,\ldots,\ell$. Hence, $u_{0,i} \geq 0$ and $\bf u_0=(u_{0,1},\ldots,u_{0,\ell})\in\cH^{\phi_n}$.

We claim that
\begin{equation*}
u_{0,i}\neq 0\qquad\forall i=1,\ldots,\ell.
\end{equation*}
To prove this claim assume, by contradiction, that $u_{0,i}=0$. Following the argument in the proof of Theorem \ref{thm:splitting} we see that, after passing to a subsequence, there exist $\xi_k\in\rn$, $C_0>0$ and $\delta>0$ such that
\begin{equation} \label{eq:nonnull2}
\int_{B_{C_0+1}(\xi_k)}|u_{k,i}|^{2p}>\delta>0\qquad\forall k\in\n,
\end{equation}
where, either $\xi_k=0$, or $\xi_k = (\zeta_k,0) \in \mathbb{C} \times \mathbb{R}^{N-2}$ and $\zeta_k\to\infty$. Since $u_{k,i} \to 0$ strongly in $L^2_\mathrm{loc}(\rn)$, equation \eqref{eq:nonnull2} implies that $\xi_k\neq 0$. Now, as in Case 2 of Theorem \ref{thm:splitting}, we set
$$w_{k,i}(x):=u_{k,i}(x+\xi_{k}),\qquad i=1,\ldots,\ell,$$
and we take a subsequence satisfying $w_{k,i}\rh w_i$ weakly in $H^1(\rn)$, $w_{k,i}\to w_i$ in $L^{2p}_\mathrm{loc}(\rn)$ and $w_{k,i}\to w_i$ a.e. in $\rn$. Hence, $w_i\in H^1(\rn)^{O(N-2)}$, $w_i\geq 0$ and following the proof of \eqref{eq:norms} we obtain
\begin{equation} \label{eq:norms2}
\|u_{k,i}\|_V^2=\Big\|u_{k,i}-\sum_{j=0}^{n-1}(w_{i}\circ\alpha^{-j })(\,\cdot\,-\alpha^j\xi_k)\Big\|_V^2+n\|w_{i}\|_{V_\infty}^2+o(1).
\end{equation}
Furthermore, following the proof of \eqref{eq:solution} we derive
\begin{equation*} 
\irn(\nabla w_i\cdot\nabla v+V_\infty w_iv)=\irn|w_i|^{2p-2}w_iv+\beta_k\sum\limits_{\substack{j=1\\j\neq i}}^\ell\irn|w_j|^p|w_i|^{p-2}w_iv
\end{equation*}
for every $v\in H^1(\rn)^{O(N-2)}$, and taking $v=w_i$ we get
\begin{equation*}
\|w_i\|_{V_\infty}^2=\irn|w_i|^{2p}+\beta_k\sum\limits_{\substack{j=1\\j\neq i}}^\ell\irn|w_j|^p|w_i|^{p}\leq\irn|w_i|^{2p}.
\end{equation*}
Since, by \eqref{eq:nonnull2},
\begin{equation*}
\int_{B_{C_0+1}(0)}|w_{k,i}|^{2p}\geq\int_{B_{C_0+1}(\xi_k)}|u_{k,i}|^{2p}\geq\delta>0,
\end{equation*}
we see that $w_i\neq 0$. Hence, there exists $t\in(0,1]$ such that $\|tw_i\|_{V_\infty}^2=\irn|tw_i|^{2p}$,  and \eqref{eq:norms2} yields
\begin{align*}
n\mathfrak{c}_\infty\leq n\,\frac{p-1}{2p}\|tw_i\|_{V_\infty}^2\leq n\,\frac{p-1}{2p}\|w_i\|_{V_\infty}^2\leq\frac{p-1}{2p}\|u_{k,i}\|_V^2.
\end{align*}
As a consequence,
\begin{align*}
\ell n\mathfrak{c}_\infty\leq\frac{p-1}{2p}\sum_{i=1}^\ell\|u_{k,i}\|_V^2,
\end{align*}
contradicting \eqref{eq:estimate}. This shows that $u_{0,i}\neq 0$, as claimed.

As $(\cJ_k^{\phi_n})'(\bf u_k)=0$, $u_{k,i}\geq 0$, $u_{0,i}\geq 0$ and $\beta_k<0$, we have that
\begin{align*}
\langle u_{k,i},u_{0,i}\rangle_V&=\irn |u_{k,i}|^{2p-2}u_{k,i}u_{0,i}+\beta_k\sum\limits_{\substack{j=1\\j\neq i}}^\ell\irn|u_{k,j}|^p|u_{k,i}|^{p-2}u_{k,i}u_{0,i} \\
&\leq \irn |u_{k,i}|^{2p-2}u_{k,i}u_{0,i},
\end{align*}
and passing to the limit we obtain $\|u_{0,i}\|_V^2\leq |u_{0,i}|_{2p}^{2p}$. Hence, there exists $s\in(0,1]$ such that $\|su_{0,i}\|_V^2=|su_{0,i}|_{2p}^{2p}$ and we have that
\begin{equation} \label{eq:1}
\mathfrak{c}^{G_n}\leq\frac{p-1}{2p}\|su_{0,i}\|_V^2\leq\frac{p-1}{2p}\|u_{0,i}\|_V^2\leq\liminf_{k\to\infty}\frac{p-1}{2p}\|u_{k,i}\|_V^2,
\end{equation}
with $\mathfrak{c}^{G_n}$ as in \eqref{eq:frak_c}. We claim that these are equalities.

To prove this claim, let $v_k\in\cM^{G_n}$ be such that $J(v_k)=\frac{p-1}{2p}\|v_k\|_V^2\to \mathfrak{c}^{G_n}$. Set $u_{k,1}:=v_k$ and define $u_{k,j+1}$ as in \eqref{eq:symmetries} for $j=1,\ldots,\ell-1$. Set $\bf u_k=(u_{k,1},\ldots,u_{k,\ell})$. Since  $(v_k)$ is bounded in $H^1(\rn)$ and $\beta_k\to 0$, we have that
$$\lim_{k\to\infty}\beta_k\irn|u_{k,j}|^p|u_{k,i}|^p=0\qquad\text{for every \ }i,j,$$
so, by Proposition \ref{prop:nehari}$(d)$, for $k$ large enough there exists $s_k\in(0,\infty)$ such that $s_k\bf u_k\in\cN_k^{\phi_n}$ and $s_k\to 1$ as $k\to\infty$. Thus,
\begin{equation} \label{eq:2}
c_k^{\phi_n}\leq\cJ^{\phi_n}(s_k\bf u_k)=\frac{p-1}{2p}\sum_{i=1}^\ell\|s_ku_{k,i}\|_V^2=\frac{p-1}{2p}\ell s_k^2\|v_k\|_V^2\longrightarrow\ell \mathfrak{c}^{G_n}.
\end{equation}
Combining \eqref{eq:1} and \eqref{eq:2} we see that $s=1$, thus $u_{0,i}\in\cM^{G_n}$, that $u_{k,i}\to u_{0,i}$ strongly in $H^1(\rn)$ and that
$$J(u_{0,i})=\mathfrak{c}^{G_n}=\frac{1}{\ell}\,c_k^{\phi_n}<n\,\mathfrak{c}_\infty.$$
This completes the proof.
\end{proof}
\smallskip

Now we turn to the case $\beta\to-\infty$. For the proof of Theorem \ref{thm:partition} we need the following result.

\begin{lemma} \label{lem:Linfty}
Let $\beta_k<0$ and $(u_{k,1},\ldots,u_{k,\ell})$ be a solution to \eqref{eq:system} with $\beta=\beta_k$ such that $u_{k,i}\to u_{\infty,i}$ strongly in $H^1(\rn)$ for every $i=1,\ldots,\ell$. Then $(u_{k,i})$ is uniformly bounded in $L^\infty(\rn)$.
\end{lemma}

\begin{proof}
Let $s\geq 0$ and assume that $u_{k,i}\in L^{2(s+1)}(\rn)$ for every $k\in\n$. Fix $L>0$ and define $w_{k,i}:=u_{k,i}\min\{u^s_{k,i},L\}$. Then,
\begin{align} \label{eq:nabla}
&\irn|\nabla w_{k,i}|^2 \leq (1+s)\irn\nabla u_{k,i}\cdot\nabla(u_{k,i}\min\{u^{2s}_{k,i},L^2\})\\
&=(1+s)\Big(\irn|u_{k,i}|^{2p-2}w_{k,i}^2 + \beta\sum_{\substack{j=1 \\ j\neq i}}^\ell\irn|u_{k,j}|^p|u_{k,i}|^{p-2}w_{k,i}^2-\irn Vw_{k,i}^2\Big)\nonumber\\
&\leq (1+s)\irn|u_{k,i}|^{2p-2}w_{k,i}^2.\nonumber
\end{align}
On the other hand, for any $K>0$ we have that
\begin{align*}
&\irn|u_{k,i}|^{2p-2}w_{k,i}^2 \\
&\leq\irn(|u_{k,i}|^{2p-2}-|u_{\infty,i}|^{2p-2})w_{k,i}^2 + \int_{|u_{\infty,i}|^{2p-2}\geq K}|u_{\infty,i}|^{2p-2}w_{k,i}^2 + K\irn w_{k,i}^2 \\
&\leq\Big||u_{k,i}|^{2p-2}-|u_{\infty,i}|^{2p-2}\Big|_{2p}^{2p-2}|w_{k,i}|_{2p}^2 + \Big(\int_{|u_{\infty,i}|^{2p-2}\geq K}|u_{\infty,i}|^{2p}\Big)^\frac{p-1}{p}|w_{k,i}|_{2p}^2 \\
&\qquad + K|w_{k,i}|_2^2.
\end{align*}
As $u_{k,i}\to u_{\infty,i}$ strongly in $H^1(\rn)$, choosing $k_0>0$ and $K$ sufficiently large, we get that
\begin{equation} \label{eq:nonl}
\irn|u_{k,i}|^{2p-2}w_{k,i}^2\leq\frac{1}{2}|w_{k,i}|_{2p}^2 + K|w_{k,i}|_2^2\qquad\text{for every \ }k\geq k_0.
\end{equation}
Since $H^1(\rn)$ is continuously embedded into $L^{2p}(\rn)$, we derive from \eqref{eq:nabla} and \eqref{eq:nonl} that, for every $k\in\n$,
$$|w_{k,i}|_{2p}^2\leq K_s|w_{k,i}|_2^2,$$
for some constant $K_s$ independent of $L$, and letting $L\to\infty$ we get
\begin{equation*}
|u_{k,i}|_{2p(s+1)}^{2(s+1)}=|u_{k,i}^{s+1}|_{2p}^2\leq K_s|u_{k,i}^{s+1}|_2^2=K_s|u_{k,i}|_{2(s+1)}^{2(s+1)}.
\end{equation*}
As $(u_{k,i})$ is uniformly bounded in $L^2(\rn)$, iterating this inequality starting with $s=0$ and using interpolation, we conclude that $(u_{k,i})$ is uniformly bounded in $L^q(\rn)$ for any $q\in[2,\infty)$ and each $i=1,\ldots,\ell$. This implies that
$$f_{k,i}:=|u_{k,i}|^{2p-2}u_{k,i}+\beta\sum_{j\neq i}|u_{k,j}|^p|u_{k,i}|^{p-2}u_{k,i}$$
is uniformly bounded in $L^q(\rn)$ for any $q\in[2,\infty)$. Then, by the Calderón-Zygmund inequality, $(u_{k,i})$ is uniformly bounded in $W^{2,q}(\rn)$ for every $q\in[2,\infty)$ and, choosing $q$ large enough, we derive from the Sobolev embedding theorem that $(u_{k,i})$ is uniformly bounded in $L^\infty(\rn)$, as claimed.
\end{proof}

\begin{proof}[Proof of Theorem \ref{thm:partition}]
$(i):$ \ As before, we write $\cJ_k^{\phi_n}$ and $\cN_k^{\phi_n}$ for the functional and the Nehari set associated to the system \eqref{eq:system} with $\beta=\beta_k$, and set
$$c_k^{\phi_n}:=\inf_{\cN_k^{\phi_n}}\cJ_k^{\phi_n}.$$
Arguing as in the proof of Theorem \ref{thm:beta_to_0}, we see that, after passing to a subsequence, $u_{k,i} \rh u_{\infty,i}$ weakly in $H^1(\rn)$, $u_{k,i} \to u_{\infty,i}$ strongly in $L^2_\mathrm{loc}(\rn)$ and $u_{k,i} \to u_{\infty,i}$ a.e. in $\rn$, for each $i=1,\ldots,\ell$. Hence, $u_{\infty,i} \geq 0$ and $\bf u_\infty=(u_{\infty,1},\ldots,u_{\infty,\ell})\in\cH^{\phi_n}$, so $\bf u_\infty$ satisfies \eqref{eq:symmetries}. We also get that 
\begin{equation*}
u_{\infty,i}\neq 0 \quad \text{and} \quad \|u_{\infty,i}\|_V^2\leq |u_{\infty,i}|_{2p}^{2p}\qquad\forall i=1,\ldots,\ell.
\end{equation*}
Furthermore, as $(\cJ_k^{\phi_n})'(\bf u_k)=0$, we have that, for every $j\neq i$,
\begin{align*}
0\leq\irn |u_{k,j}|^p|u_{k,i}|^p\leq \frac{|u_{k,i}|^{2p}_{2p}}{-\beta_k}\leq \frac{C}{-\beta_k}.
\end{align*}
As $\beta_k\to-\infty$, passing to the limit and using Fatou's lemma we obtain
$$0 \leq \irn |u_{\infty,j}|^p|u_{\infty,i}|^p \leq \liminf_{k \to \infty} \irn |u_{k,j}|^p|k_{n,i}|^p = 0.$$
This implies that $u_{\infty,j} u_{\infty,i} = 0$ a.e. in $\rn$ whenever $i\neq j$. 

Let $s\in(0,1]$ be such that $\|su_{\infty,i}\|_V^2=|su_{\infty,i}|_{2p}^{2p}$. Then, $s\bf u_\infty\in\cW^n_\ell$ and using \eqref{eq:estimate} we get
\begin{align*}
\widehat{c}^{\phi_n}&\leq\frac{p-1}{2p}\sum_{i=1}^\ell\|su_{\infty,i}\|_V^2\leq\frac{p-1}{2p}\sum_{i=1}^\ell\|u_{\infty,i}\|_V^2 \\
&\leq\frac{p-1}{2p}\sum_{i=1}^\ell\liminf_{k\to\infty}\|u_{k,i}\|_V^2\leq \widehat{c}^{\phi_n}.
\end{align*}
This proves that $s=1$, \ $\bf u_\infty\in\cW^n_\ell$, \ $u_{k,i}\to u_{\infty,i}$ strongly in $H^1(\rn)$ \ and
\begin{equation}  \label{eq:weak_optimal}
\widehat{c}^{\phi_n}=\frac{p-1}{2p}\sum_{i=1}^\ell\|u_{\infty,i}\|_V^2.
\end{equation}
Finally, as \ $\lim\limits_{k\to\infty}\|u_{k,i}\|_V^2=\|u_{\infty,i}\|_V^2=|u_{\infty,i}|_{2p}^{2p}=\lim\limits_{k\to\infty}|u_{k,i}|_{2p}^{2p}$, \ from
$$\lim_{k\to\infty}\|u_{k,i}\|_V^2=\lim_{k\to\infty}|u_{k,i}|_{2p}^{2p}+\lim_{k\to\infty}\beta_k\sum\limits_{\substack{j=1\\j\neq i}}^\ell\irn|u_{k,j}|^p|u_{k,i}|^p$$
we obtain
\begin{equation*}
\irn \beta_k u_{k,j}^p u_{k,i}^p\to 0 \text{ as } k\to \infty\quad \text{whenever } i\neq j.
\end{equation*}

$(ii):$ \ It follows from Lemma \ref{lem:Linfty} and \cite[Theorem B.2]{cpt} that $(u_{k,i})$ is uniformly bounded in $\cC^{0,\alpha}(K)$ for each compact subset $K$ of $\rn$ and $\alpha\in(0,1)$. So from the Arzelà-Ascoli theorem we get that $u_{\infty,i}\in\cC^0(\rn)$. Therefore $\o_i:=\{x\in\rn:u_{\infty,i}(x)>0\}$ is open. Since $u_{\infty,i}u_{\infty,j}=0$ if $i\neq j$ and $\bf u_\infty$ satisfies \eqref{eq:symmetries}, we have that $\o_i\cap\o_j=\emptyset$ if $i\neq j$ and the $\ell$-tuple $(\o_1,\ldots,\o_\ell)$ satisfies $(S_1)$ and $(S_2)$. Thus, it is an $(n,\ell)$-pinwheel partition.

Since $\bf u_\infty\in\cW^n_\ell$, we have that $u_{\infty,i}$ belongs to the Nehari manifold $\cM_{\o_i}$ defined in \eqref{eq:bdd_nehari}. Therefore, $\frac{p-1}{2p}\|u_{\infty,i}\|_V^2\geq \mathfrak{c}_{\o_i}$. Equality must hold true as, otherwise, there would exist $v_1\in\cM_{\o_1}$ such that $\frac{p-1}{2p}\|u_{\infty,1}\|_V^2>\frac{p-1}{2p}\|v_1\|_V^2\geq \mathfrak{c}_{\o_i}$ and, defining $v_{j+1}$ as in \eqref{eq:symmetries}, we would have that $(v_1,\ldots,v_\ell)\in\cW_\ell^n$ and
$$\frac{p-1}{2p}\sum_{i=1}^\ell\|v_i\|_V^2<\frac{p-1}{2p}\sum_{i=1}^\ell\|u_{\infty,i}\|_V^2=\widehat{c}^{\phi_n}$$
by \eqref{eq:weak_optimal}, which is a contradiction. This shows that $u_{\infty,i}$ is a least energy solution of \eqref{eq:equation_bdd} in $\o_i$. Now, since $\frac{p-1}{2p}\|u_{\infty,i}\|_V^2=\mathfrak{c}_{\o_i}$, we get
\begin{align*}
\inf_{(\t_1,\ldots,\t_\ell)\in\cP^{n}_\ell}\sum_{j=1}^\ell\mathfrak{c}_{\t_j}\leq\sum_{j=1}^\ell\mathfrak{c}_{\o_j}=\widehat{c}^{\phi_n}\leq\inf_{(\t_1,\ldots,\t_\ell)\in\cP^{n}_\ell}\sum_{j=1}^\ell\mathfrak{c}_{\t_j}.
\end{align*}
This shows that $(u_{\infty,1},\ldots,u_{\infty,\ell})$ is an optimal $(n,\ell)$-pinwheel partition.

$(iii):$ \ This is a local statement. As mentioned above, $(u_{k,i})$ is uniformly bounded in $\cC^{0,\alpha}(\o)$ for each open subset $\o$ compactly contained in $\rn$ and $\alpha\in(0,1)$. So, from the Arzelà-Ascoli theorem, we get that $u_{k,i}\to u_{\infty,i}$ in $\cC^{0,\alpha}(\o)$. Thus, all hypotheses of \cite[Theorem C.1]{cpt} are satisfied and $(iii)$ follows from it.

$(iv):$ \ Let $G_{2n}$ be the group defined in Example \ref{example} with $\ell=2$, and let $\tau_n:G_{2n}\to\z_2:=\{1,-1\}$ be the homomorphism given by $\tau_n(\e^{2\pi\i/2n})=-1$ and $\tau_n(\theta)=1$ for every $\theta\in O(N-2)$. A solution to the Schrödinger equation \eqref{eq:equation} satisfying
\begin{equation}\label{eq:nodal}
u(gx)=\tau_n(g)u(x)\qquad\text{for all \ }g\in G_{2n}, \ x\in\rn,
\end{equation}
is a critical point of the functional $J:H^1(\rn)^{\tau_n}\to\r$ defined by \eqref{eq:J} on the space
$$H^1(\rn)^{\tau_n}:=\{u\in H^1(\rn):u\text{ satisfies }\eqref{eq:nodal}\}.$$
The nontrivial ones belong to the Nehari manifold
$$\cM^{\tau_n}:=\{u\in H^1(\rn)^{\tau_n}:u\neq 0, \ \|u\|_V^2 =|u|_{2p}^{2p}\},$$
which is a natural constraint for $J$. Note that every nontrivial function satisfying \eqref{eq:nodal} is nonradial and changes sign.

There is a one-to-one correspondence
$$\cW^n_2\to\cM^{\tau_n},\qquad (u_1,u_2) \mapsto u_1-u_2,$$
whose inverse is $u\mapsto (u^+,-u^-)$, with $u^+:=\max\{u,0\}$ and $u^-:=\min\{u,0\}$, \ satisfying 
$$\frac{p-1}{2p}\Big(\|u_1\|^2_V+\|u_2\|^2_V\Big)=J(u_1-u_2).$$
Therefore,
\begin{align*}
J(u_{\infty,1}-u_{\infty,2})&=\frac{p-1}{2p}\Big(\|u_{\infty,1}\|^2_V+\|u_{\infty,2}\|^2_V\Big) \\
&=\inf_{(u_1,u_2)\in\cW^n_2}\frac{p-1}{2p}\Big(\|u_1\|^2_V+\|u_2\|^2_V\Big)=\inf_{u\in\cM^{\tau_n}}J(u).
\end{align*}
This shows that \ $u_{\infty,1}-u_{\infty,2}$ \ is a least energy solution to \eqref{eq:equation} and \eqref{eq:nodal}.
\end{proof}

\appendix

\section{An auxiliary result}

\begin{lemma} \label{lem:A1}
Assume $v_k\rh v$ weakly in $H^1(\rn)$, $\xi_k\in\rn$ satisfies $|\xi_k|\to\infty$ and $V\in\cC^0(\rn)$ satisfies $(V_2)$. Set $V_k(x):=V(x+\xi_k)$. Then,
$$\lim_{k\to\infty}\|v_k\|_{V_k}^2-\lim_{k\to\infty}\|v_k-v\|_{V_k}^2=\|v\|_{V_\infty}^2.$$
\end{lemma}

\begin{proof}
As $v_k\rh v$ weakly in $H^1(\rn)$ one has
\begin{align*}
\|v\|_{V_\infty}^2+o(1)&=\|v_k\|_{V_\infty}^2-\|v_k-v\|_{V_\infty}^2\\
&=\|v_k\|_{V_k}^2-\|v_k-v\|_{V_k}^2+2\irn(V_\infty-V_k)v_kv-\irn(V_\infty-V_k)v^2.
\end{align*}
Given $\eps>0$, choose $R>0$ large enough so that
$$\int_{\rn\smallsetminus B_R}|V_\infty-V_k||v|^2\leq 2\sup_{x\in\rn}V(x)\int_{\rn\smallsetminus B_R}|v|^2<\frac{\eps}{2}.$$
Now, take $k_0$ such that
$$|V_\infty - V(x+\xi_k)|<\frac{\eps}{2|v|_2^2}=:\delta\quad\text{for every \ }x\in B_R\text{ \ and \ }k\geq k_0.$$
Then, for $k\geq k_0$ we have that
$$\irn|V_\infty-V_k||v|^2\leq\int_{B_R}|V_\infty-V_k||v|^2+\int_{\rn\smallsetminus B_R}|V_\infty-V_k||v|^2<\eps$$
and
\begin{align*}
\irn|(V_\infty-V_k)v_kv|&\leq\left(\irn|V_\infty-V_k||v_k|^2\right)^\frac{1}{2}\left(\irn|V_\infty-V_k||v|^2\right)^\frac{1}{2}\\
&\leq C\sqrt{\eps}.
\end{align*}
This completes the proof.
\end{proof}

\bigskip

\begin{flushleft}
\textbf{Mónica Clapp}\\
Instituto de Matemáticas\\
Universidad Nacional Autónoma de México \\
Campus Juriquilla\\
Boulevard Juriquilla 3001\\
76230 Querétaro, Qro., Mexico\\
\texttt{monica.clapp@im.unam.mx} 
\medskip

\textbf{Angela Pistoia}\\
Dipartimento di Metodi e Modelli Matematici\\
La Sapienza Università di Roma\\
Via Antonio Scarpa 16 \\
00161 Roma, Italy\\
\texttt{angela.pistoia@uniroma1.it} 
\end{flushleft}
	

\begin{thebibliography}{99}

\bibitem{cl} Caffarelli, L. A.; Lin, Fang-Hua: Singularly perturbed elliptic systems and multi-valued harmonic functions with free boundaries. J. Amer. Math. Soc. 21 (2008), no. 3, 847–862. 

\bibitem{clll} Chang, Shu-Ming; Lin, Chang-Shou; Lin, Tai-Chia; Lin, Wen-Wei: Segregated nodal domains of two-dimensional multispecies Bose-Einstein condensates. Phys. D 196 (2004), no. 3-4, 341–361. 

\bibitem{cmp} Clapp, Mónica; Maia, Liliane A.; Pellacci, Benedetta: Positive multipeak solutions to a zero mass problem in exterior domains. Commun. Contemp. Math. 23 (2021), no. 2, Paper No. 1950062, 22 pp. 

\bibitem{cpt} Clapp, Mónica; Pistoia, Angela; Tavares, Hugo: Yamabe systems, optimal partitions, and nodal solutions to the Yamabe equation. Preprint arXiv:2106.00579.

\bibitem{cs} Clapp, Mónica; Salazar, Dora:   Multiple sign changing solutions of nonlinear elliptic problems in exterior domains. Adv. Nonlinear Stud. 12 (2012), no. 3, 427–443.

\bibitem{csz} Clapp, Mónica; Szulkin, Andrzej: A simple variational approach to weakly coupled competitive elliptic systems. NoDEA Nonlinear Differential Equations Appl. 26 (2019), no. 4, Paper No. 26, 21 pp. 

\bibitem{cw}  Clapp, Mónica; Weth, Tobias: Multiple solutions of nonlinear scalar field equations. Comm. Partial Differential Equations 29 (2004), no. 9-10, 1533–1554.

\bibitem{ctv1} Conti, M.; Terracini, S.; Verzini, G.: Nehari's problem and competing species systems. Ann. Inst. H. Poincaré C Anal. Non Linéaire 19 (2002), no. 6, 871–888.

\bibitem{ctv2}  Conti, Monica; Terracini, Susanna; Verzini, Gianmaria: A variational problem for the spatial segregation of reaction-diffusion systems. Indiana Univ. Math. J. 54 (2005), no. 3, 779–815.

\bibitem{gg} Gao, Fengshuang; Guo, Yuxia: Multiple solutions for a nonlinear Schrödinger systems. Commun. Pure Appl. Anal. 19 (2020), no. 2, 1181–1204.

\bibitem{liweiwu}  Li, Tuoxin;  Wei, Juncheng;  Wu, Yuanze: Infinitely many nonradial positive solutions for multi-species nonlinear Schrödinger systems in $\mathbb R^N$. Preprint. arXiv:2210.03330.

\bibitem{nttv} Noris, Benedetta; Tavares, Hugo; Terracini, Susanna; Verzini, Gianmaria: Uniform Hölder bounds for nonlinear Schrödinger systems with strong competition. Comm. Pure Appl. Math. 63 (2010), no. 3, 267–302.

\bibitem{pw} Peng, Shuangjie; Wang, Zhi-Qiang: Segregated and synchronized vector solutions for nonlinear Schrödinger systems. Arch. Ration. Mech. Anal. 208 (2013), no. 1, 305–339.

\bibitem{pv} Pistoia, Angela; Vaira, Giusi: Segregated solutions for nonlinear Schrödinger systems with weak interspecies forces. Comm. PDE (to appear). arXiv:2203.01551. 

\bibitem{sttz} Soave, Nicola; Tavares, Hugo; Terracini, Susanna; Zilio, Alessandro: Hölder bounds and regularity of emerging free boundaries for strongly competing Schrödinger equations with nontrivial grouping. Nonlinear Anal. 138 (2016), 388–427.

\bibitem{tt} Tavares, Hugo; Terracini, Susanna: Regularity of the nodal set of segregated critical configurations under a weak reflection law. Calc. Var. Partial Differential Equations 45 (2012), no. 3-4, 273–317. 

\bibitem{w} Willem, Michel: Minimax theorems. Progress in Nonlinear Differential Equations and their Applications 24. Birkhäuser Boston, Inc., Boston, MA (1996).
		
		
\end{thebibliography}
\end{document}